\documentclass[11pt]{article}
\usepackage{amsmath,amsfonts,amsthm,amssymb, mathtools}
\usepackage{mathrsfs, graphicx,color,latexsym, tikz, calc,
}
\usepackage[colorlinks,bookmarksopen,bookmarksnumbered,citecolor=blue, linkcolor=red, urlcolor=blue]{hyperref}
\usepackage{enumitem}
\usepackage{authblk}
\usetikzlibrary{shadows}
\usetikzlibrary{patterns,arrows,decorations.pathreplacing}

\makeatletter
\def\leftharpoonfill@{\arrowfill@\leftharpoonup\relbar\relbar}
\def\rightharpoonfill@{\arrowfill@\relbar\relbar\rightharpoonup}
\newcommand\rbjt{\mathpalette{\overarrow@\rightharpoonfill@}}
\newcommand\lbjt{\mathpalette{\overarrow@\leftharpoonfill@}}
\makeatother

\newtheorem{theorem}{Theorem}
\newtheorem{lemma}{Lemma}

\newtheorem{conjecture}{Conjecture}

\newtheorem{claim}{Claim}[section]

\begin{document}

\title{\bf \Large  Oriented Ramsey numbers of some sparse graphs}

\author{Junying Lu$^1$}
\author{Yaojun Chen$^{1,2,}$\footnote{Corresponding author. yaojunc@nju.edu.cn.}}
\affil{\small $^1$School of Mathematics, Nanjing University, Nanjing 210093, China\\
$^2$School of Mathematics and Statistics, Minnan Normal University, Zhangzhou, Fujian 363000, PR China}
\date{ }

\maketitle

\begin{abstract}
Let $H$ be an oriented graph without directed cycle.
The oriented Ramsey number of  $H$, denoted by $\rbjt{r}(H)$, is the smallest integer $N$ such that every tournament on $N$ vertices contains a copy of $H$. Rosenfeld (JCT-B, 1974) conjectured that $\rbjt{r}(H)=|H|$ if $H$ is a cycle of sufficiently large order, which was confirmed for $|H|\geq 9$ by Zein recently, and so does if $H$ is a path. Note that $\rbjt{r}(H)=|H|$ implies any tournament contains $H$ as a spanning subdigraph, it is interesting to ask when $\rbjt{r}(H)=|H|$ for $H$ being a sparse oriented graph. S\'os (1986) conjectured this is true if $H$ is a directed path plus an additional edge containing the origin of the path as one end, which was confirmed by Petrovi\'{c} (JGT, 1988). In this paper, we show that $\rbjt{r}(H)=|H|$ for $H$ being an oriented graph obtained by identifying a vertex of an antidirected cycle with one end of a directed path. Some other oriented Ramsey numbers for oriented graphs with one cycle are also discussed.

\end{abstract} 

\baselineskip=0.202in

\section{Introduction}
A digraph $D$ is a pair $D=(V(D),E(D))$, where $V(D)$ is a set of vertices and $E(D)$ is the set of arcs of $D$ such that $E\subseteq (V\times V)\setminus \{(v,v): v \in V \}$. An \emph{oriented graph} $D$ is a digraph where
$(u, v) \in E(D)$ implies $(v, u) \notin E(D)$ for every $u, v \in V(D)$. For a digraph $D$, if $(u,v)$ is an arc, we say that $u$ \emph{dominates} $v$ and write $u\to v$. If $v_1\to v_2$ for any $v_1\in V_1$, then we write $V_1\to v_2$ and the notation $v_1\to V_2$ is defined similarly. If $v_1\to v_2$ for any $v_1\in V_1$ and $v_2\in V_2$, then we write $V_1\to V_2$ or $V_2\gets V_1$. For any $W\subseteq V(D)$, we denote by $D[W]$ the subdigraph induced by $W$ in $D$, and  $D-W=D[V(D)\setminus W]$. The \emph{dual} digraph of $D$ is the digraph $-D$ on the same set of vertices such that $x\to y$ is an arc of $-D$ if and only if $y\to x$ is an arc of $D$. Let $v$ be a vertex of $D$. The \emph{out-neigbourhood} of $v$, denoted by $N_D^+(v)$, is the set of vertices $w$ such that $v\to w$. The \emph{in-neigbourhood} of $v$, denoted by $N_D^-(v)$, is the set of vertices $w$ such that $w\to v$. The \emph{out-degree} $d_D^+(v)$ (resp. the \emph{in-degree} $d_D^-(v)$) is $|N_D^+(v)|$ (resp. $|N_D^-(v)|$). Compared to the well known directed path (cycle), the \emph{antidirected paths (cycles)} are the oriented paths (cycles) in which every vertex has either in-degree $0$ or out-degree $0$ (in other words, two consecutive edges are oriented in opposite ways).

A tournament is an orientation of a complete graph. A tournament is regular if each vertex has the same out-degree. The \emph{oriented Ramsey number} of an oriented graph $H$, denoted by $\rbjt{r}(H)$, is the smallest integer $N$ such that every tournament on $N$ vertices contains a copy of $H$. Because of transitive tournaments, which are acyclic, $\rbjt{r}(H)$ is finite if and only if $H$ is acyclic. Note that $\rbjt{r}(D)=\rbjt{r}(-D)$ for any acyclic oriented graph $D$. Indeed, if any tournament $T$ of order $n$ contains $D$, then $-T$ contains $D$ and so $T$ contains $-D$. The oriented Ramsey numbers of oriented paths and non-directed cycles were widely studied.

It started with R\'{e}dei's theorem \cite{Redei} which states that the oriented Ramsey number of $\vec{P}_n$, the directed path on $n$ vertices, is $n$. Later on, in 1971, Gr\"{u}nbaum \cite{Grunbaum} proved that the oriented Ramsey number of an antidirected path of order $n$ is $n$ unless $n=3$ (in which case it is not contained in the tournament which is a directed $3$-cycle) or $n=5$ (in which case it is not contained in the regular tournament of order $5$) or $n=7$ (in which case it is not contained in the Paley tournament of order $7$). In the same year, Rosenfeld \cite{Rosenfeld} gave an easier proof and conjectured that there is a smallest integer $N>7$ such that $\rbjt{r}(P)=|P|$ for every oriented path of order at least $N$. The condition $N>7$ results from Gr\"{u}nbaum's counterexamples. Several papers gave partial answers to this conjecture \cite{Alspach, Forcade, Straight} until Rosenfeld's conjecture was verified by Thomason, who proved in \cite{Thomason} that $N$ exists and is less than $2^{128}$. Finally, Havet and Thomass\'{e} \cite{Havet}, showed that $\rbjt{r}(P)=|P|$ for every oriented path $P$ except the antidirected paths of order $3$, $5$ and $7$.

Concerning the oriented cycles, Gr\"{u}nbaum \cite{Grunbaum} conjectured that the oriented Ramsey number of the antidirected cycle on $n\ge 10$ vertices is $n$. Let $AC_{2k}$ denote the antidirected cycle on $2k$ vertices. In 1974, Rosenfeld \cite{Rosenfeld2} proved the conjecture for large $n$ and obtained the following.
\begin{theorem}[Rosenfeld, \cite{Rosenfeld2}]\label{thm-R}
$\rbjt{r}(AC_{2k})=2k$ for $k\ge 14$.
\end{theorem}
Rosenfeld \cite{Rosenfeld2} also conjectured the existence of some integer $N$, such that every tournament on $n$ vertices, $n > N$, contains any oriented Hamiltonian cycle, except possibly the directed ones. Thomason \cite{Thomason} showed that any tournament $T$ of order $n\ge 2^{128}+1$ is pancyclic, that is, $T$ contains every non-directed oriented cycle $C$ with $3\le |C|\le n$. In 1999, Havet \cite{Havet1} proved that every tournament of order $n \ge 68$ contains any oriented Hamiltonian cycle, except possibly the directed ones. Recently, Zein \cite{Zein} showed that, with exactly $35$ exceptions, every tournament of order $n \ge 3$ is pancyclic. In particular, any tournament contains each Hamiltonian non-directed cycle with 30 exceptions, all of order less than $9$. 

The above results imply that any tournament contains each oriented path and non-directed cycle as a spanning subdigraph provided the order is large enough. However, this property does not hold for some sparse oriented graphs, even for oriented trees. For example, if $H$ is the out-star on $n$ vertices (whose edges are oriented from the central vertex to each of the $n-1$ leaves), then $\rbjt{r}(H)=2|H|-2$. It is therefore natural to ask when $\rbjt{r}(H)=|H|$ for $H$ being a sparse oriented graph.

For $3\le i\le n$, let $H(n,i)$ denote the oriented graph with vertex set $\{1,2,\ldots,n\}$ and arc set $\{(1,i),(j,j+1)\colon 1\le j\le n-1\}$. At the Sixth Yugoslav Seminar on Graph Theory in Zagreb (1986), S\'{o}s posed the following conjecture.
\begin{conjecture}[S\'{o}s, 1986]
$\rbjt{r}(H(n,i))=n$ for each $n$ and $i~(4\le i\le n-1)$.
\end{conjecture}
Petrovi\'{c} \cite{Petrovic} completely resolved this conjecture. He showed that if $3\le i\le n$, then any tournament $T_n$ contains a copy of $H(n,i)$ unless $i=3$ or $i=5$ and $T_n$ belongs to a certain class of exceptional tournaments. In fact, $H(n,i)$ can be obtained by identifying one end of a directed path $\vec{P}_{n-i+1}$ with a vertex of a specific non-directed cycle of length $i$. This leads us to the question of whether $\rbjt{r}(H)=|H|$ for an oriented graph $H$ of such type,  except for a few small exceptional oriented graphs.

The \emph{blocks} of an oriented path (resp. cycle) are the maximal subdipaths of this path (resp. cycle). It is clear that the underlying graph of $H(n,i)$ is unicyclic and the length of the largest block of the cycle attains the maximum. 

Motivated by S\'os' conjecture, we are interested in the oriented Ramsey numbers of other oriented graphs whose underlying graphs are unicyclic. In particular, we focus on the cases when the length of the largest block of the cycle attained the minimum (which is 1), that is, the antidirected cycle, or the cycle has exactly two blocks, that is, $C(p,q)$, which is obtained from a directed cycle of length $p+q$ by changing the orientation of $p$ consecutive edges. It should be noted that antidirected cycles  are highly symmetric and difficult to deal with in studying whether a tournament contains a Hamiltonian cycle of such type.

 For $k,\ell\ge 1$, we denote the oriented graph, which is obtained by identifying an end $u$ of directed path $\vec{P}_{\ell+1}$ with a vertex $v\in V(AC_{2k})$, by 
\[\begin{cases}
Q^+(2k,+\ell) & \text{if } d_{\vec{P}_{l+1}}^+(u)=1 \text{ and } d_{AC_{2k}}^+(v)=2;\\
Q^-(2k,+\ell) & \text{if } d_{\vec{P}_{l+1}}^+(u)=1 \text{ and } d_{AC_{2k}}^-(v)=2;\\
Q^+(2k,-\ell) & \text{if } d_{\vec{P}_{l+1}}^-(u)=1 \text{ and } d_{AC_{2k}}^+(v)=2;\\
Q^-(2k,-\ell) & \text{if } d_{\vec{P}_{l+1}}^-(u)=1 \text{ and } d_{AC_{2k}}^-(v)=2.
\end{cases}\]
If the sign is omitted, it is assumed positive. Note that the dual of $Q^+(2k,-\ell)$ is $Q^-(2k,+\ell)$, and the dual of $Q^-(2k,-\ell)$ is $Q^+(2k,+\ell)$, where $Q^\pm(2k,\ell)$ are shown in Fig. \ref{fig-1}. 
\begin{figure}[htb]\label{fig-1}
	\centering
	\includegraphics[width=0.6\linewidth]{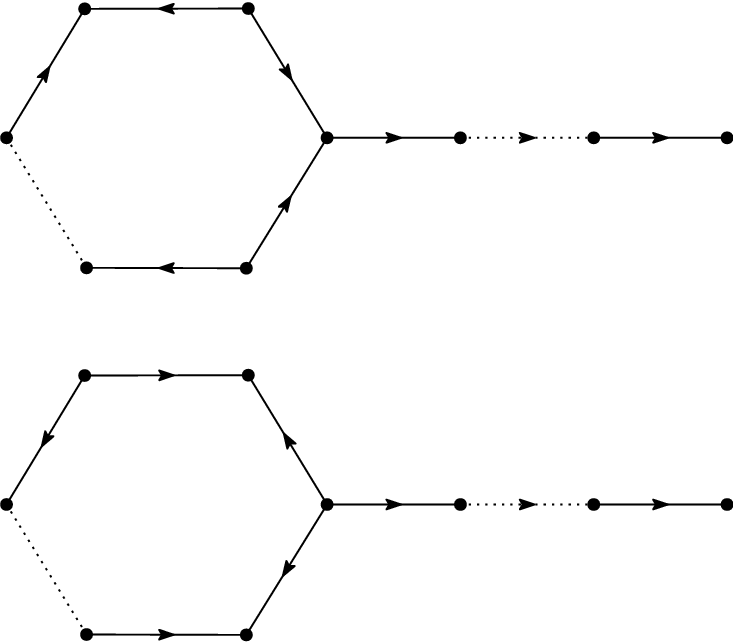}
 \put(-130,110){$Q^-(2k,\ell)$}
 \put(-130,-10){$Q^+(2k,\ell)$}
 \caption{$Q^\pm(2k,\ell)$}
\end{figure}
The first main result of this paper is on $\rbjt{r}(Q^\pm(2k,\ell))$ as below.

\begin{theorem}\label{thm-1}
For $k\ge 54$,  $\rbjt{r}(Q^\pm(2k,\ell))=2k+\ell$.
\end{theorem}

Clearly, Theorem \ref{thm-1} tells us $\rbjt{r}(H)=|H|$ for $H=Q^\pm(2k,\ell)$. Suppose $H$ is the oriented graph obtained by identifying one end of a directed path with a vertex of a non-directed cycle, it seems that the oriented Ramsey number of $H$ behaves like the non-directed cycle in $H$, that is, $\rbjt{r}(H)=|H|$ except in the case where the length of the non-directed cycle is small.

Before presenting our second result, we introduce some necessary notations. An \emph{in-arborescence} (resp. \emph{out-arborescence}) is an oriented tree in which all arcs are oriented towards (resp. away from) a fixed vertex called the root. An \emph{arborescence} is either an in-arborescence or an out-arborescence. A directed path is a special arborescence.
As a generalization, we consider the oriented graph obtained by adding a forward (resp. backward) arc from a vertex of an antidirected cycle to the root of an out-arborescence (resp. in-arborescence). By duality, we treat only the case when the arborescence is out-arborescence. For an oriented tree, the \emph{leaves} of it are the vertices having exactly one neighbor. There are two kinds of leaves: \emph{in-leaves} which have out-degree 1 and in-degree 0, and \emph{out-leaves} which have out-degree 0 and in-degree 1. Let $A$ be an out-arborescence with $\ell$ vertices and $a$ out-leaves. Denote by $ACA^+(2k;\ell, a)$ (resp. $ACA^-(2k;\ell, a)$) the oriented graph obtained by adding a forward arc from a vertex $v$ of an antidirected cycle $AC_{2k}$ to the root of $A$, where $d_{AC_{2k}}^+(v)=2$ (resp. $d_{AC_{2k}}^-(v)=2$). It is not difficult to see that $ACA^\pm(2k;\ell,1)=Q^\pm(2k,\ell)$ for $\ell\ge 2$ and $ACA^\pm(2k;1,0)=Q^\pm(2k,1)$.
\begin{theorem}\label{thm-4}
For $k\ge 25$ and $\ell>a\ge 1$, $\rbjt{r}(ACA^\pm(2k;\ell,a))\le 2k+\ell+a$.
\end{theorem}
However, we don't know the exact value of $\rbjt{r}(ACA^\pm(2k;\ell,a))$. It remains intriguing to determine the value of $\rbjt{r}(ACA^\pm(2k;\ell,a))$.

Finally, let $CP(p,q;\ell)$ be the oriented graph obtained by identifying the end $u$ of a directed path $\vec{P}_{\ell+1}$ with a vertex $v$ of $C(p,q)$ satisfied $d_{\vec{P}_{\ell+1}}^+(u)=1$ and $d_{C_{p,q}}^+(v)=0$. Note that $CP(1,i-1;n-i)=H(n,i)$. It remains interesting to consider $p,q\ge 2$. In this paper, as a special case, we determine the value of $\rbjt{r}(CP(2,2;n-4))$ and obtain the following.

\begin{theorem}\label{thm-5}
For $n\ge 4$, $\rbjt{r}(CP(2,2;n-4))= n+1$.
\end{theorem}

\section{Preliminaries}
In this section, we will introduce some results which will be used later. Firstly, as a preparation to prove Theorems \ref{thm-1} and \ref{thm-4}, we introduce some concepts about the median order and give several basic properties.

Let $\sigma=(v_1,v_2,\ldots,v_n)$ be an ordering of the vertices of a digraph $D$. An arc $(v_i,v_j)$ is \emph{forward} (according to $\sigma$) if $i<j$ and \emph{backward} (according to $\sigma$) if $j<i$. A \emph{median order} of $D$ is an ordering of the vertices of $D$ with the maximum number of forward arcs, or equivalently the minimum number of backward arcs. Some basic properties of  median orders of tournaments are as follows.
\begin{lemma}[Dross and Havet \cite{Dross}]\label{lem-1}
Let $T$ be a tournament and $(v_1,v_2,\ldots,v_n)$ a median order of $T$. Then, for any two indices $i,j$ with $1\le i<j\le n$:\\
{\rm (P1)} $(v_i,v_{i+1},\ldots,v_j)$ is a median order of the induced tournament $T[\{v_i,v_{i+1},\ldots,v_j\}]$.\\
{\rm (P2)} $v_i$ dominates at least half of the vertices $v_{i+1},v_{i+2},\ldots,v_j$, and  $v_j$ is dominated by at least half of the vertices $v_i,v_{i+1},\ldots,v_{j-1}$. In particular, each vertex $v_\ell$, $1\le \ell\le n-1$, dominates its successor $v_{\ell+1}$.
\end{lemma}
A \emph{local median order} is an ordering of the vertices of $D$ that satisfies property (P2). It is clear that a median order of a digraph $D$ is a local median order of $D$.

Now, we list some results on the oriented Ramsey numbers of oriented cycles and trees.


\begin{lemma}[Dross and Havet \cite{Dross}]\label{lem-5}
Let $A$ be an out-arborescence with $n$ vertices, $k$ out-leaves and root $r$, let $T$ be a tournament on $m=n+k-1$ vertices, and let $(v_1,v_2,\ldots,v_m)$ be a local median order of $T$. There is an embedding $\phi$ of $A$ into $T$ such that $\phi(r)=v_1$.
\end{lemma}
An ADH path (cycle) is an antidirected Hamiltonian path (cycle). If $v\to u\gets \cdots$ is an ADH path in $T_n$, then $v$ is called a \emph{starting vertex}, and if $v\gets u\to \cdots$ is an ADH path in $T_n$, then $v$ is a \emph{terminating vertex}. 
 
\begin{lemma}[Rosenfeld \cite{Rosenfeld}]\label{lem-2-6}
(a) If $T$ is a tournament with an odd number of vertices and $T$ has an ADH path, then $T$ has a double vertex, i.e., a vertex $v$ such that $T$ has an ADH path with $v$ as starting vertex and an ADH path with $v$ as terminating vertex.

(b) If $n\ge 10$ is an even integer, then for every vertex $v\in V(T_n)$ there is an ADH path on $T_n$ having $v$ as an end-vertex.
\end{lemma}
Denote the transitive tournament of order $n$ by $TT_n$. Our proof depends heavily on the existence of large transitive tournaments and we shall use the following results.
\begin{lemma}[Sanchez-Flores \cite{Sanchez}]\label{lem-2-7}
Let $k$ and $n$ be positive integers with $k\ge 7$ and $n\ge 54\cdot2^{k-7}$. Every tournament $T_n$ contains a $TT_k$.
\end{lemma}

\begin{lemma}[Rosenfeld \cite{Rosenfeld2}]\label{lem-2-8} Let $TT_n$ be a transitive tournament on $\{1,2,\ldots,n\}$ with arc set $\{(i,j)\colon 1\leq i<j\leq n\}$.

(a) If $n$ is even, then $TT_n$  has an ADH path starting at $i~(i\ne n)$ and terminating at $j$ except for the following cases: (i) $j=1$; (ii) $i=1,j=2~(n>2)$; (iii) $i=n-1,j=n~(n>2)$.

(b) If $n$ is odd, then $TT_n$ has an ADH path with $i,j$ as starting vertices if $i,j\ne n$ and for $n>3$, $\{i,j\}\ne \{n-2,n-1\}$; as terminating vertices if $i,j\ne 1$ and for $n>3$, $\{i,j\}\ne \{2,3\}$.
\end{lemma}
\section{Proofs of Theorems \ref{thm-1} and \ref{thm-4}}
The main task of this section is to give the proofs of Theorems \ref{thm-1} and \ref{thm-4}.

In order to prove Theorem \ref{thm-1}, we need the following lemma.
\begin{lemma} \label{lem-3-4}
Let $T$ be a regular tournament on $2k+1$ vertices with $k\ge 9$. If $T$ has a $TT_m$ for
$m\ge 8$, then $T$ contains $Q^\pm(2k,1)$. 
\end{lemma}

To make the arguments easier to follow, we postpone the proof of Lemma \ref{lem-3-4} until the end of this section.

Let $P=(x_1,\ldots,x_n)$ be a path. We say that $x_1$ is the \emph{origin} of $P$ and $x_n$ is the \emph{terminus} of $P$. If $x_1\to x_2$, $P$ is an \emph{outpath}, otherwise $P$ is an \emph{inpath}. The \emph{directed outpath} of order $n$ is the path $(x_1,\ldots,x_n)$ in which $x_i\to x_{i+1}$ for all $i$, $1\le i\le n-1$; the dual notion is \emph{directed inpath}.
\begin{proof}[\bfseries{Proof of Theorem \ref{thm-1}}]
Let $n=2k+\ell$ and $T$ be a tournament on $n$ vertices. Suppose $\ell=1$. If $T$ is a non-regular tournament of order $n$, let $v\in V(T)$ be a vertex with maximum in-degree. Since $k\ge 54$, $T-\{v\}$ contains a copy of $AC_{2k}$ by Theorem \ref{thm-R}. Note that $d_{AC_{2k}}^-(v)=d_T^-(v)\ge k+1$, $T$ contains $Q^\pm(2k,1)$. By Lemma \ref{lem-2-7}, $T$ contains $TT_8$. If $T$ is regular, then by Lemma \ref{lem-3-4}, $T$ contains $Q^\pm(2k,1)$.

Suppose $\ell\ge 2$. Let $\sigma=(v_1,v_2,\ldots,v_n)$ be a median order of $T$. By Lemma \ref{lem-1}, we can see that $P=(v_{2k+2},v_{2k+3},\ldots, v_n)$ is a directed outpath of length $\ell-2$. Set $A=\{v_1,\ldots,v_{2k+1}\}$. By Lemma \ref{lem-1}, we have $d_A^-(v_{2k+2})\ge k+1$. Let $X\subseteq A$ and $|X|=2k$ such that $X$ contains as many vertices in $N_A^-(v_{2k+2})$ as possible. Since $k\ge 54$, by Theorem \ref{thm-R}, $T[X]$ contains an antidirected cycle $C$ of length $2k$. Since $d_C^-(v_{2k+2})\ge k+1$, there exist $x,y\in V(C)$ such that $d_C^+(x)=d_C^-(y)=2$ and $(x,v_{2k+2}),(y,v_{2k+2})\in E(T)$. Therefore, $T$ contains $Q^+(2k,\ell-1)$ and $Q^-(2k,\ell-1)$. Suppose to the contrary that $T$ contains neither $Q^+(2k,\ell)$ nor $Q^-(2k,\ell)$. Let $A\setminus X=\{v\}$. Obviously, $v\to v_n$. We claim that $v\to v_j$ for any $2k+2\le j\le n-1$. If there exists $j$ with $2k+2\le j\le n-1$ such that $v_j \to v$, denote by $j_0$ the largest $j$ such that $v_{j_0}\to v$, then $v\to v_{j_0+1}$ and thus we obtain a directed outpath $(v_{2k+2},\ldots,v_{j_0},v,v_{j_0+1},\ldots,v_{n})$. It follows that $T$ contains $Q^\pm(2k,\ell)$, a contradiction. In particular, $v\to v_{2k+2}$. By the choice of $X$, we have $X\to v_{2k+2}$ and thus $A\to v_{2k+2}$. Now we can choose $X$ to be any $2k$-subset of $A$, and hence we have $v_i\to \{v_{2k+2},\ldots,v_n\}$ for any $1\le i\le 2k+1$. It follows that $A\to V(T)\setminus A$. Since $T[A]$ contains $Q^\pm(2k,1)$ by the arguments before, we can see that $T$ contains $Q^\pm(2k,\ell)$, again a contradiction. Therefore, the result follows.
\end{proof}

\begin{proof}[\bfseries{Proof of Theorem \ref{thm-4}}]
Let $n=2k+\ell+a$ and $A$ be an out-arborescence with $\ell$ vertices, $a$ out-leaves and root $r$. Suppose $\sigma=(v_1,v_2,\ldots,v_n)$ be a median order of a tournament $T_n$. By Lemma \ref{lem-1}(P1), $(v_{2k+2},v_{2k+3},\ldots,v_n)$ is a median order of the induced subtournament $T[\{v_{2k+2},v_{2k+3},\ldots,v_n\}]$. Therefore, there is an embedding $\phi$ of $A$ in $T_n$ such that $\phi(r)=v_{2k+2}$ due to Lemma \ref{lem-5}. Set $B=\{v_1,\ldots,v_{2k+1}\}$. By Lemma \ref{lem-1}, we have $d_B^-(v_{2k+2})\ge k+1$. Let $X\subseteq B$ such that $|X|=2k$ and $X$ contains as many vertices in $N_B^-(v_{2k+2})$ as possible. Therefore, $T_n[X]$ contains an antidirected cycle $C$ with length $2k$. Since $d_C^-(v_{2k+2})\ge k+1$, there exist $x,y\in V(C)$ such that $d_C^+(x)=d_C^-(y)=2$ and $(x,v_{2k+2}),(y,v_{2k+2})\in E(T_n)$. Therefore, $T_n$ contains $ACA^+(2k;\ell,a)$ and $ACA^-(2k;\ell,a)$.    
\end{proof}

We now give the proof of Lemma \ref{lem-3-4} in details.
\begin{proof}[\bfseries{Proof of Lemma \ref{lem-3-4}}]
Let $TT_m$ be a maximum transitive subtournament with vertex set $\{1,2,\dots,m\}$ and arc set $\{(i,j):1\leq i<j\leq m\}$ and let $T^*=T-TT_m=T[\{a_1,a_2,\ldots,a_s\}]$. Since $T$ is regular, we have $m\le k+1$ and $s\ge k$.
{\flushleft\bf Case 1.} $m$ is even.
\vskip 2mm
In this case, $s$ is odd and $s\ge 9$. By Lemma \ref{lem-2-6}(a), assume that $a_1$ is a double vertex in tournament $T^*$.  Since $m\ge 8$, we have  $m-5\geq 3$, and hence there exist two vertices $x,y\in \{3,\ldots,m-3\}$ such that $\{x,y\}\to a_1$ or $\{x,y\}\gets a_1$.

If $\{x,y\}\to a_1$, let  $a_1\gets a_2\cdots a_{s-1}\to a_s$ be an ADH path in $T^*$. Since $m$ is maximum, there is some vertex $t\in V(TT_m)$ such that $t\to a_s$. If $t\notin \{m-1,m\}$, let $z=\min\{\{x,y\}\setminus \{t\}\}$, then $z,t\not=m-1$ and $\{z,t\}\not=\{m-3,m-2\}$. By Lemma \ref{lem-2-8}(b), $TT_m-\{m\}$ has an ADH path with $z$ and $t$ as starting vertices and hence $z\to a_1\gets a_2\cdots a_{s-1}\to a_s\gets t\to \cdots \gets z$ is an ADH cycle in $T-\{m\}$. Since $TT_m-\{m\} \to m$, $T$ contains $Q^\pm(2k,1)$. Therefore, we may assume that $a_s\to \{1,2,\ldots,m-2\}$. Again, by the maximality of $m$, there is some vertex $t_1\in V(TT_m)$ such that $a_{s-1}\to t_1$. By symmetry of $x$ and $y$, assume that $t_1\not=x$. If $t_1\notin\{m-1,m\}$, then by Lemma \ref{lem-2-8}(a), $TT_m-\{m,t_1\}$ has an ADH path $x\to \cdots \to t_2$ for some $t_2\in V(TT_m)\setminus \{x,t_1,m-1,m\}$, and hence 
$x\to a_1\gets a_2\cdots a_{s-1}\to t_1\gets a_s\to t_2\gets \cdots \gets x$ is an ADH cycle in $T-\{m\}$. It follows that $T$ contains $Q^\pm(2k,1)$.  Since $T'=T[\{a_s,1,\ldots,m-2\}]$ is a transitive subtournament in $T$, by Lemma \ref{lem-2-8}(b), $T'$ contains an ADH path $1\to \cdots\gets x$. If $t_1=m-1$, then $x\to a_1\gets a_2 \cdots a_{s-1}\to t_1\gets 1\to \cdots \gets x$ is an ADH cycle in $T-\{m\}$, and so $T$ contains $Q^\pm(2k,1)$. If $t_1=m$, then we may assume that $\{1,2,\ldots,m-1\}\to a_{s-1}$. It follows that $T[\{1,\ldots,m-1,a_{s-1},m\}]$ is a transitive subtournament on $m+1$ vertices, which contradicts that $m$ is maximum.

If $\{x,y\}\gets a_1$, let $a_1\to a_2 \cdots a_{s-1}\gets a_s$ be an ADH path in $T^*$. Since $m$ is maximum, there is some vertex $t\in V(TT_m)$ such that $t\gets a_s$. If $t\notin \{1,m\}$, let $z=\max\{\{x,y\}\setminus \{t\}\}$, then $z,t\not=1$ and $\{z,t\}\not=\{2,3\}$. By Lemma \ref{lem-2-8}(b), $TT_m-\{m\}$ has an ADH path with $z$ and $t$ as terminating vertices and hence $z\gets a_1\to a_2 \cdots a_{s-1}\gets a_s\to t\gets \cdots \to z$ is an ADH cycle in $T-\{m\}$. Since $TT_m-\{m\} \to m$, $T$ contains $Q^\pm(2k,1)$. Therefore, we may assume that $a_s\gets \{2,\ldots,m-2,m-1\}$. Again, by the maximality of $m$, there is some vertex $t_1\in V(TT_m)$ such that $a_{s-1}\gets t_1$. By symmetry of $x$ and $y$, assume that $t_1\ne x$. If $t_1\notin\{1,m\}$, then by Lemma \ref{lem-2-8}(a),  $TT_m-\{m,t_1\}$  have an ADH path $x\gets \cdots \gets t_2$ for some $t_2\in V(TT_m)\setminus \{1,x,t_1,m\}$, and hence $x\gets a_1\to a_2 \cdots a_{s-1}\gets t_1\to a_s\gets t_2\to \cdots \to x$ is an ADH cycle in $T-\{m\}$. It follows that $T$ contains $Q^\pm(2k,1)$. Since $T'=T[\{2,\ldots,m-1,a_s\}]$ is a transitive subtournament in $T$, by Lemma \ref{lem-2-8}(b), $T'$ contains an ADH path $m-1\gets \cdots\to x$. If $t_1=1$, then $x\gets a_1\to a_2 \cdots a_{s-1}\gets t_1\to m-1\gets \cdots \to x$ is an ADH cycle in $T-\{m\}$, and so $T$ contains $Q^\pm(2k,1)$. 
Now, we are left to consider $t_1=m$.
If $a_s\to m$, then since $TT_m-\{m-1\}$ has an ADH path $m\gets \cdots \to x$ by Lemma \ref{lem-2-8}(b), $x\gets a_1\to a_2 \cdots a_s\to m\gets \cdots \to x$ is an ADH cycle in $T-\{m-1\}$. Note that  $\{1,\ldots,m-2\}\to m-1$, one can see that $T$ contains $Q^\pm(2k,1)$. Therefore, we may assume that $m\to a_s$. In such a case, by Lemma \ref{lem-2-8}(a), there is some vertex $z\in \{2,\ldots,m-3\}$ such that $TT_m-\{m-1,m\}$ has an ADH path $z\to \cdots \to x$. Thus, $x\gets a_1\to \cdots \to a_{s-1}\gets t_1\to a_s\gets z\to \cdots \to x$ is an ADH cycle in $T-\{m-1\}$. Hence $T$ contains $Q^\pm(2k,1)$.
{\flushleft\bf Case 2.} $m$ is odd.
\vskip 2mm
In this case, $s$ is even and $s\ge 10$. Let $i_0=\frac{m-1}{2}$. To complete the proof of this case, we need the following claim.

\begin{claim}\label{cla-3-1}
If $TT_m-\{i_0\}\to v$ or $v\to TT_m-\{i_0\}$ for some $v\in V(T^*)$, then $T$ contains $Q^\pm(2k,1)$.
\end{claim}
\begin{proof}
 Set $T''=T[V(TT_m)\cup \{v\}-\{i_0\}]$. Since $m$ is maximum, $T''$ is a maximal transitive subtournament. Relabel $V(T'')$ as $\{1',2',\ldots,m'\}$ such that $i'\to j'$ for $i<j$. Note that $\{2',\ldots,(i_0-1)'\}\to i_0$ and $i_0\to \{(i_0+1)',\ldots,(m-1)'\}$. Since $s$ is even and $s\ge 10$, by Lemma \ref{lem-2-6}(b), $T-T''$ has an ADH path with $i_0$ as an end vertex. Firstly, let the ADH path be 
\[i_0\gets b_2\to \cdots\gets b_s~(\{b_2,\ldots,b_s\}=V(T^*)\setminus\{v\}).
\]
Since $m$ is maximum, there is a vertex $t\in V(T'')$ such that $b_s\to t$. If $t\notin\{1',m'\}$, then by Lemma \ref{lem-2-8}(a), one can choose a vertex $x\in \{2',\ldots,(i_0-1)'\}$ and $x\ne t$ such that $T''-\{m'\}$ have an ADH path $x\to \cdots \to t$, and so $x\to i_0\gets b_2\cdots b_s\to t\gets \cdots\gets x$ is an ADH cycle in $T-\{m'\}$. It indicates that $T$ contains $Q^\pm(2k,1)$. Thus, we may assume that $\{2',\ldots,(m-1)'\}\to b_s$. Let $t'\in V(T'')$ such that $t'\to b_{s-1}$. If $t'\notin\{1',m'\}$, then by Lemma \ref{lem-2-8}(b), there is an ADH path $x'\to \cdots \gets x$ in $T''-\{t',m'\}$ such that $x\in \{2',\ldots,(i_0-1)'\}$ and $x'\in \{2',\ldots,(m-2)'\}\setminus \{t',x\}$. Thus, $x\to i_0\gets b_2 \cdots b_{s-1}\gets t'\to b_s\gets x'\to \cdots \gets x$ is an ADH cycle in $T-\{m'\}$. It follows that $T$ contains $Q^\pm(2k,1)$. If $t'=1'$, then $y\to i_0\gets b_2 \cdots  b_{s-1}\gets t'\to (m-1)'\gets (m-2)'\to b_s\gets y'\to \cdots \gets y$, where $y'\to \cdots\gets y$ is an ADH path in $T''-\{1',(m-2)',(m-1)',m'\}$, is an ADH cycle in $T-\{m'\}$. Therefore, $T$ contains $Q^\pm(2k,1)$. 
Now, we are left to consider $t'=m'$.
If $b_s\to m'$, then $(i_0-1)'\to i_0\gets b_2 \cdots b_s\to m'\gets \cdots \gets (i_0-1)'$ is an ADH cycle in $T-\{(m-1)'\}$, where $m'\gets \cdots \gets (i_0-1)'$ is an ADH path in $T''-\{(m-1)'\}$. Thus, $T$ contains $Q^\pm(2k,1)$. Therefore, we have $m'\to b_s$. By Lemma \ref{lem-2-8}(b), there exists $z\in \{3',\ldots,(m-3)'\}$ such that $T''-\{(m-1)',m'\}$ has an ADH path $z\to \cdots \gets 2'$. Thus, $2'\to i_0\gets b_2 \cdots b_{s-1}\gets t'\to b_s\gets z\to \cdots \gets 2'$ is an ADH cycle in $T-\{(m-1)'\}$. Hence $T$ contains $Q^\pm(2k,1)$.

Next, let the ADH path be 
\[i_0\to b_2\gets \cdots\to b_s~(\{b_2,\ldots,b_s\}=V(T^*)\setminus\{v\}).
\]
Since $m$ is maximum, there is a vertex $t\in V(T'')$ such that $b_s\gets t$. If $t\notin\{(m-1)',m'\}$, then by Lemma \ref{lem-2-8}(a), one can choose a vertex $x\in \{(i_0+1)',\ldots,(m-2)'\}$ and $x\ne t$ such that $T''-\{m'\}$ have an ADH path $x\gets \cdots \gets t$, and thus  $x\gets i_0\to b_2\cdots b_s\gets t\to \cdots\to x$ is an ADH cycle in $T-\{m'\}$. It indicates that $T$ contains $Q^\pm(2k,1)$. Thus, we may assume that $\{1',\ldots,(m-2)'\}\gets b_s$. Let $t'\in V(T'')$ such that $t'\gets b_{s-1}$. If $t'\notin\{(m-1)',m'\}$, then there is an ADH path $x'\gets \cdots \to x$ in $T''-\{t',m'\}$ such that $x\in \{(i_0+1)',\ldots,(m-1)'\}$ and $x'\in \{1',\ldots,(m-2)'\}\setminus \{x,t'\}$. It follows that $x\gets i_0\to b_2 \cdots b_{s-1}\to t'\gets b_s\to x'\gets \cdots \to x$ is an ADH cycle in $T-\{m'\}$. 
Thus, $T$ contains $Q^\pm(2k,1)$. If $t'=(m-1)'$, then $y\gets i_0\to b_2 \cdots  b_{s-1}\to (m-1)'\gets 1'\to 2'\gets b_s\to y'\gets \cdots \to y$, where $y'\gets \cdots\to y$ is an ADH path in $T''-\{1',2',(m-1)',m'\}$, is an ADH cycle in $T-\{m'\}$. Therefore, $T$ contains $Q^\pm(2k,1)$. 
If $t'=m'$, then we may assume that $\{1',2',\ldots,(m-1)'\}\to b_{s-1}$. It follows that $T[\{1',2',\ldots,(m-1)',b_{s-1},m'\}]$ is a transitive subtournament on $m+1$ vertices, which contradicts that $m$ is maximum.
\end{proof}

By Lemma \ref{lem-2-6}(a), there is a double vertex in $T[\{a_1,\ldots,a_s,i_0\}]$.
\vskip 2mm
\noindent{\bf Subcase 2.1.} $i_0$ is not the double vertex.
\vskip 2mm
Assume that $a_1$ is a double vertex in $T[\{a_1,\ldots,a_s,i_0\}]$. Since $m\ge 9$, there exists $\{x,y\}\subseteq \{3,\ldots,i_0-1,i_0+1,\ldots,m-3\}$ such that $a_1\gets \{x,y\}$ or $a_1\to \{x,y\}$.

If $a_1\gets \{x,y\}$, then let  $a_1\gets b_2\to \cdots \to b_{s+1}$ be an ADH path in $T[V(T^*)\cup\{i_0\}]$, where $\{b_2,\ldots,b_{s+1}\}=\{a_2,\ldots,a_s,i_0\}$. If $b_{s+1}=i_0$, then by Lemma \ref{lem-2-8}(a), there is an ADH path in $TT_m-\{m\}$ with $x$ as starting vertex and $i_0$ as terminating vertex. It follows that $x\to a_1\gets b_2\cdots b_{s+1}\gets \cdots \gets x$ is an ADH cycle in $T-\{m\}$, and so $T$ contains $Q^{\pm}(2k,1)$. Therefore, we may assume $b_{s+1}\ne i_0$. Since $m$ is maximum, there is a vertex $t\in V(TT_m)$ such that $b_{s+1}\gets t$. If $t\notin\{i_0,m-1,m\}$,  we choose $x$ or $y$, say $x$, such that
$x\ne t$ and $\{x,t\}\ne \{m-3,m-2\}$. By Lemma \ref{lem-2-8}(b), $TT_m-\{i_0,m\}$ has an ADH path with $x$ and $t$ as starting vertices and thus $x\to a_1\gets b_2 \cdots  b_{s+1}\gets t\to \cdots \gets x$ is an ADH cycle in $T-\{m\}$. It indicates that $T$ contains $Q^\pm(2k,1)$, and so we can assume $b_{s+1}\to V(TT_m)\setminus \{i_0,m-1,m\}$. Consider the vertex $b_s$. If $b_s=i_0$, then $x\to a_1\gets b_2 \cdots b_s\to m-2\gets b_{s+1}\to z\gets \cdots \gets x$, where $z\gets \cdots \gets x$ is an ADH path in $TT_m-\{m,m-2,i_0\}$, is an ADH cycle in $T-\{m\}$. Hence $T$ contains $Q^\pm(2k,1)$ and hence we assume $b_s\ne i_0$. Since $m$ is maximum, there is a vertex $t_1\in V(TT_m)$ such that $b_s\to t_1$. By symmetry of $x$ and $y$, assume that $t_1\ne x$. If $t_1\notin \{i_0,m-1,m\}$, then we can always find a vertex $t_2\in V(TT_m)\setminus\{t_1,x,i_0,m-1,m\}$ such that  $TT_m-\{t_1,m,i_0\}$ has an ADH path $t_2\gets  \cdots \gets x$ due to Lemma \ref{lem-2-8}(a). Hence $x\to a_1\gets b_2 \cdots  b_s\to t_1\gets b_{s+1}\to t_2\gets \cdots \gets x$ is an ADH cycle in $T-\{m\}$ and so $T$ contains $Q^\pm(2k,1)$. Note that $T'=T[\{b_{s+1},1,\ldots,i_0-1,i_0+1,\ldots,m-2\}]$ is a transitive subtournament. If $t_1\in \{m-1,m\}$, then by Lemma \ref{lem-2-8}(a), $T'$ contains an ADH path $1\to \cdots \gets x$. Hence $x\to a_1\gets b_2 \cdots  b_s\to t_1\gets 1\to \cdots \gets x$ is an ADH cycle in $T-\{m\}$ or $T-\{m-1\}$ and so $T$ contains $Q^\pm(2k,1)$. Now, we may assume that $V(TT_m)\setminus \{i_0\}\to b_s$. By Claim \ref{cla-3-1}, the result follows.

If $a_1\to \{x,y\}$, then let $a_1\to b_2\gets \cdots \gets b_{s+1}$ be an ADH path in $T[V(T^*)\cup\{i_0\}]$, where $\{b_2,\ldots,b_{s+1}\}=\{a_2,\ldots,a_s,i_0\}$. Similar argument yields that either $T$ contains $Q^\pm(2k,1)$ or there exists $v\in V(T^*)$ such that $v\to V(TT_m)\setminus \{i_0\}$. If the latter occurs, by Claim \ref{cla-3-1}, $T$ also contains $Q^\pm(2k,1)$.
\vskip 2mm
\noindent{\bf Subcase 2.2.} $i_0$ is the double vertex.
\vskip 2mm
Let $i_0\gets a_1 \cdots a_{s-1}\to a_s$ be an ADH path in $T[V(T^*)\cup \{i_0\}]$. Since $m$ is maximum, let $t\in V(TT_m)$ be a vertex such that $t\to a_s$. If $t\notin \{i_0,m-1,m\}$, then by Lemma \ref{lem-2-8}(a), there exists an ADH path starting at $t$ and terminating at $i_0$ in $TT_m-\{m\}$. Hence $i_0\gets a_1 \cdots a_s\gets t \to \cdots \to i_0$ is an ADH cycle in $T-\{m\}$. It yields that $T$ contains $Q^\pm(2k,1)$. Thus we may  assume $a_s\to V(TT_m)\setminus \{i_0,m-1,m\}$. Consider the vertex $a_{s-1}$. Let $t_1\in V(TT_m)$ such that $a_{s-1}\to t_1$. If $t_1\notin \{i_0,m-1,m\}$, then by Lemma \ref{lem-2-8}(b) and the choice of $i_0$, we can find a suitable vertex $t_2\in V(TT_m)\setminus \{t_1,i_0,m-1,m\}$ such that $t_2\gets \cdots \to i_0$ is an ADH path in $TT_m-\{t_1,m\}$. Hence $i_0\gets a_1 \cdots a_{s-1}\to t_1\gets a_s\to t_2\gets \cdots \to i_0$ is an ADH cycle in $T-\{m\}$. It follows that $T$ contains $Q^\pm(2k,1)$. Note that $T'=T[\{a_s,1,\ldots,i_0-1,i_0+1,\ldots,m-2\}]$ is a transitive subtournament on $m-2$ vertices. If $t_1\in \{m-1,m\}$, then there is a suitable vertex $z\in V(T')$ such that $1\to \cdots \gets z$ is an ADH path in $T'$ and $z\to i_0$. Hence $i_0\gets a_1 \cdots a_{s-1}\to t_1\gets 1\to \cdots\gets z\to i_0$ is an ADH cycle in $T-\{m\}$ or $T-\{m-1\}$. Since $\{1,\ldots,m-2\}\to \{m-1,m\}\setminus \{t_1\}$, $T$ contains $Q^\pm(2k,1)$. Now, we may assume $V(TT_m)\setminus \{i_0\}\to a_{s-1}$. By Claim \ref{cla-3-1}, $T$ contains $Q^\pm(2k,1)$.
\end{proof}

\section{Proof of Theorem \ref{thm-5}}
We first show that $\rbjt{r}(CP(2,2;n-4))\ge n+1$. Let $T=T_4^* \to T_{n-4}$, where $T_4^*=\vec{C}_3\to T_1$ or $T_1\to \vec{C}_3$. We claim $T$ contains no $CP(2,2;n-4)$. Suppose to the contrary that $T$ contains a copy of $CP(2,2;n-4)$, denoted by $H$. Let $E(H)=\{(v_1,v_2),(v_1,v_3),(v_2,v_4),(v_3,v_4)\}\cup \{(v_i,v_{i+1}): 4\le i\le n-1\}$. Note that $d_H^+(v_i)=1$ for $2\le i\le n-1$. Let $i_0$ be the smallest index such that $v_j\in V(T_{n-4})$ for any $j\ge i_0$. Obviously, $i_0\ge 5$. By the construction of $T$, we have $\{v_1,v_2,\ldots,v_{i_0-1}\}\subseteq V(T_4^*)$. It implies that $i_0=5$. However, $T_4^*$ contains no $C(2,2)$. Therefore, $H\not \cong CP(2,2;n-4)$, a contradiction.

Now, we proceed to prove $\rbjt{r}(CP(2,2;n-4))=n+1$ by induction on $n$. The following lemma indeed establishes the result for the case when $n=4$, and acts as the induction basis.
\begin{lemma}\label{lem-6}
Every tournament $T_5$ contains $C(2,2)$. 
\end{lemma}
\begin{proof}
Suppose to the contrary that there exists a tournament $T_5$ containing no $C(2,2)$. Let $V(T_5)=\{u_1,\ldots,u_5\}$. Note that $T_5$ contains a copy of a transitive triangle, say $C$, with $E(C)=\{(u_1,u_2),(u_2,u_3),(u_1,u_3)\}$. Assume without loss of generality that $u_4\to u_5$. Suppose that $u_4\to u_1$. If $u_5\to u_2$ or $u_5\to u_3$, then $u_1u_2u_5u_4u_1$ or $u_1u_3u_5u_4u_1$ is a $C(2,2)$, and if $u_2\to u_5$ and $u_3\to u_5$, then $u_1u_2u_5u_3u_1$ is a $C(2,2)$, a contradiction. Thus we may assume that $u_1\to u_4$. If $u_4\to u_3$, then $u_1u_2u_3u_4u_1$ is a $C(2,2)$ and so we have $u_3\to u_4$. If $u_2\to u_4$, then $u_1u_2u_4u_3u_1$
is a $C(2,2)$ and hence $u_4\to u_2$. Now, if $u_3\to u_5$, then $u_1u_3u_5u_4u_1$ is a $C(2,2)$ and if $u_5\to u_3$, then $u_2u_3u_5u_4u_2$ is a $C(2,2)$, a final contradiction. Therefore, the result follows.
\end{proof}

\begin{proof}[\bfseries{Proof of Theorem \ref{thm-5}}]
For $n=4$, the assertion follows from Lemma \ref{lem-6}. Assume now $n\ge 5$ and the assertion holds for $n-1$. By induction hypothesis, for any tournament $T_{n+1}$, its subtournament $T_n$ contains a copy of $CP(2,2;n-5)$, say $H$. Let $V(H)=\{v_1,\ldots,v_{n-1}\}$, $E(H)=\{(v_1,v_2),(v_1,v_3),(v_2,v_4),(v_3,v_4)\}\cup \{(v_i,v_{i+1}): 4\le i\le n-2\}$ and $V(T_{n+1})\setminus V(H)=\{v_n,v_{n+1}\}$. Assume without loss of generality that $v_2\to v_3$ and $v_n\to v_{n+1}$. In the following proof, suppose to the contrary that $T_{n+1}$ contains no $CP(2,2;n-4)$.

\begin{claim}\label{cla-1}
For any $v\in \{v_n,v_{n+1}\}$, $v\to \{v_4,v_5,\ldots,v_{n-1}\}$.
\end{claim}
\begin{proof}
It's easy to see that $v\to v_{n-1}$, otherwise $E(H)\cup \{(v_{n-1},v)\}$ produces a copy of $CP(2,2;n-4)$. If there exists some $i$ with $4\le i\le n-2$ such that $v_i\to v$, denote by $i_0$ the largest $i$ such that $v_{i_0}\to v$, then $v\to v_{i_0+1}$ and thus we obtain a directed outpath $(v_4,v_5,\ldots,v_{i_0},v,v_{i_0+1},\ldots,v_{n-1})$. It follows that $T_{n+1}[V(H)\cup \{v\}]$ contains $CP(2,2;n-4)$, a contradiction. Hence, $v\to \{v_4,v_5,\ldots,v_{n-1}\}$.
\end{proof}

\begin{claim}\label{cla-2}
For any $v\in \{v_n,v_{n+1}\}$, if $v_1\to v$, then $v_3\to v$; and if $v\to v_1$, then $v_2\to v$.    
\end{claim}
\begin{proof}
If $v_1\to v$ and $v\to v_3$, then we can see that $v_1vv_3v_2v_1$ is a $C(2,2)$ in which $v_3$ has out-degree 0 and $(v_3,v_4,\ldots,v_{n-1})$ is a directed outpath, which implies that $T_{n+1}$ contains a copy of $CP(2,2;n-4)$, a contradiction.

If $v\to v_1$ and $v\to v_2$, then $vv_1v_3v_2v$ is a $C(2,2)$ in which $v_3$ has out-degree 0 and $(v_3,v_4,...,v_{n-1})$ is a directed outpath, which implies that $T_{n+1}$ contains a copy of $CP(2,2;n-4)$, a contradiction.
\end{proof}

We consider the following three cases separately.
{\flushleft\bf Case 1.} $v_1\to \{v_n,v_{n+1}\}$. 
\vskip 2mm
In this case, we have $v_3\to \{v_n,v_{n+1}\}$ by Claim \ref{cla-2}. Thus, $v_1v_nv_{n+1}v_3v_1$ is a $C(2,2)$ in which $v_{n+1}$ has out-degree 0. By Claim \ref{cla-1},  $v_{n+1}\to v_4$ and so $(v_{n+1},v_4,v_5,\ldots,v_{n-1})$ is a directed outpath of length $n-4$. Therefore, $T_{n+1}$ contains a copy of $CP(2,2;n-4)$, a contradiction.
{\flushleft\bf Case 2.} $v_1\to v_n$ and $v_{n+1}\to v_1$. 
\vskip 2mm
In this case, we have $v_2\to v_{n+1}$ by Claim \ref{cla-2}. Thus, $v_1v_2v_{n+1}v_nv_1$ is a $C(2,2)$ in which $v_{n+1}$ has out-degree 0. By Claim \ref{cla-1},  $v_{n+1}\to v_4$ and so $(v_{n+1},v_4,v_5,\ldots,v_{n-1})$ is a directed outpath of length $n-4$. Therefore, $T_{n+1}$ contains a copy of $CP(2,2;n-4)$, a contradiction.
{\flushleft\bf Case 3.} $v_n\to v_1$.
\vskip 2mm
By Claim \ref{cla-2}, we have $v_2\to v_n$. Thus, $v_2v_3v_{n+1}v_nv_2$ is a $C(2,2)$ in which $v_{n+1}$ has out-degree 0 if $v_3\to v_{n+1}$, and $v_1v_3v_{n+1}v_nv_1$ is a $C(2,2)$ in which $v_3$ has out-degree 0 if $v_{n+1}\to v_3$. By Claim \ref{cla-1}, $(v_{n+1},v_4,v_5,\ldots,v_{n-1})$ is a directed outpath of length $n-4$. Clearly, $(v_3,v_4,\ldots,v_{n-1})$ is also a directed outpath of length $n-4$. Therefore, $T_{n+1}$ always contains a copy of $CP(2,2;n-4)$, a contradiction.
\end{proof}


\section*{Declarations}
The authors declare that they have no known competing financial interests or personal relationships that could have appeared to influence the work reported in this paper.
\section*{\bf\Large Availability of Data and Materials}  
Not applicable.
\section*{Acknowledgments}
This research was supported by National Key R\&D Program of China under grant number 2024YFA1013900, NSFC under grant number 12471327, and Postgraduate Research \& Practice Innovation Program of Jiangsu Province under grant number KYCX24\_0124.

\end{document}